\documentclass[]{article}
\usepackage{amsmath,amsfonts,amssymb}
\usepackage{amsthm}
\usepackage[mathfrak]{}
\usepackage{cases}
\usepackage[latin1]{inputenc}
\usepackage[T1]{fontenc}
\usepackage{verbatim}
\usepackage{graphicx}
\usepackage{float}
\usepackage{mathrsfs}
\usepackage [all]{xy}
\usepackage{color}
\usepackage{tikz}
\usepackage{comment}
\usepackage{hyperref}
\usepackage{cases}

\newtheorem{thm}{Theorem}[section]
\newtheorem{lem}[thm]{Lemma}
\newtheorem{cor}[thm]{Corollary}
\newtheorem{pro}[thm]{Proposition}
\newtheorem{defn}[thm]{Definition}

\newtheorem{rem}[thm]{Remark}

\title{The Chow ring of a sequence of point blow-ups}
\author{Daniel Camaz\'on \footnote{The author was partially supported by PGC2018-096446-B-C21}}
\date{}

\begin{document}
\maketitle

\begin{abstract}
Given a sequence of point blow-ups of smooth $n-$dimensional projective varieties $Z_{i}$ defined over an algebraically closed field $\mathit{k}$, $Z_{s}\xrightarrow{\pi_{s}} Z_{s-1}\xrightarrow{\pi_{s-1}}\cdot\cdot\cdot\xrightarrow{\pi_{2}} Z_{1}\xrightarrow{\pi_{1}} Z_{0}$, we give two presentations of the Chow ring of its sky $A^{\bullet}(Z_{s})$. The first one using the classes of the total transforms of the exceptional components as generators and the second one using the classes of the strict transforms ones. We prove that the skies of two sequences of point blow-ups of the same length have isomorphic Chow rings. Finally we give a characterization of final divisor of a sequence of point blow-ups in terms of some relations defined over the Chow group of zero-cycles of its sky $A_{0}(Z_{s})$. 
\end{abstract}

\section{Introduction}

Sequences of blow-ups of smooth varieties along smooth centers are useful for general algebraic geometric purposes, in particular, for resolution and classification of singularities. This paper is devoted to explicitly compute the Chow ring of the variety obtained after a sequence of point blow-ups. We will assume, additionally, that the center of each blow up has normal crossings with the already created exceptional divisors by the precedent blow ups. In this case, the Chow ring is generated as a $\mathbb{Z}-$algebra by the exceptional divisors and the generic hyperplane.
 
The composition of the successive blow ups of such sequences is a projective (and therefore proper) birational morphisms $Z_{s}\rightarrow Z_{0}$, where $Z_{s}$ and $Z_{0}$ are smooth algebraic varieties which are respectively called sky and ground. In this paper, we restrict to the case of point blow-ups and $Z_{0}\cong\mathbb{P}^{n}$, as it is involved in many geometric contexts, in particular, the study of algebraic curves and surfaces. 

Our main goal is to give an explicit presentation of the Chow ring of the sky of a sequence of point blow-ups as a finite type $\mathbb{Z}-$algebra. In \cite{Keel92} Keel gives a presentation of the Chow ring of a blow-up $\pi: Z_{1}\rightarrow Z_{0}$ with center $C_{1}$ a regularly embedded subvariety under the assumption that the restriction map $i_{1}^{*}: A^{\bullet}(Z_{0})\rightarrow A^{\bullet}(C_{1})$ is surjective, and no more results have been found in the literature. 

We give two presentations of the Chow ring of the sky of a sequence of point blow-ups $A^{\bullet}(Z_{s})$ by considering the total transforms and the strict transforms of the irreducible components of the exceptional divisor as generators respectively. We give polynomial generators of the ideal of relations for the total (respectively strict) exceptional components. As a corollary, we prove a surprising result: the Chow rings of the skies of two sequences of point blow-ups of the same length are isomorphic. 

An irreducible exceptional component $E$ is called final if there are not blow-ups with center lying over the point producing $E$ for the first time, see \cite{CamazonEncinas22} and Definition \ref{DefFD}. The presentation of the Chow ring in terms of the strict transform of exceptional divisors, allow us to characterize final components by polynomial conditions. This improves conditions given in \cite{CamazonEncinas22}.
\section{Preliminaries}

Fix an algebraically closed field $\mathit{k}$. Throughout this paper a variety will mean a reduced projective scheme over $\mathit{k}$.

\begin{defn}\label{DefSBU}
A sequence of blow-ups over $\mathit{k}$ is defined as a sequence of morphisms
\begin{equation*}
Z_{s}\xrightarrow{\pi_{s}} Z_{s-1}\xrightarrow{\pi_{s-1}}\cdot\cdot\cdot\xrightarrow{\pi_{2}} Z_{1}\xrightarrow{\pi_{1}} Z_{0},
\end{equation*}
where $Z_{0}$ is a smooth $n-$dimensional projective variety and
\begin{enumerate}
\item the morphism $\pi_{i}: Z_{i}\rightarrow Z_{i-1}$ is the blow-up at center $C_{i}\subset Z_{i-1}$ for $i=1,...,s$. We denote by $E_{i}^{i}$ to be the exceptional hypersurface of $\pi_{i}$, and for $j<i$ $E_{j}^{i}\subset Z_{i}$ the strcit transform of $E_{j}^{j}\subset Z_{j}$ in $Z_{i}$,
\item $codim(C_{i+1}, Z_{i})\geq 2$ for $i=0,...,s-1$,
\item the center $C_{i}$ has simple normal crossings with $\{E_{1}^{i-1}, E_{2}^{i},...,E_{i-1}^{i-1}\}$, for $i=1,...,s$.
\end{enumerate}
We will refer to $Z_{0}$ and $Z_{s}$ as the ground and the sky of the sequence of blow-ups respectively.
\end{defn}

Along the paper we fix a sequence of blow ups $(Z_{0},...,Z_{s},\pi)$ as in Definition \ref{DefSBU} and we set $\pi_{j,i}: Z_{j}\rightarrow Z_{i}$, for $j>i$, to be the composition $\pi_{j,i}=\pi_{i+1}\circ\pi_{i+2}\circ...\circ\pi_{j-1}\circ\pi_{j}$. \\

The centers $C_{i}$, in general, can have any dimension. We extend the well-known notion of proximity for point blow-ups.

\begin{defn}
We say that $C_{j}$ is proximate to $C_{i}$, and write $C_{j}\rightarrow C_{i}$ if and only if $C_{j}\subset E_{i}^{j-1}$.
\end{defn}

Note that, if $C_{j}$ is proximate to $C_{i}$ then $j>i$.

\begin{rem}
For $j>i$ we denote by $E_{i}^{j*}$ the total transform of $E_{i}^{i}$ by the morphism $\pi_{j,i}: Z_{j}\rightarrow Z_{i}$. By an abuse of notation $E_{i}^{i*}=E_{i}^{i}$. Note that by definition of the total transform, one has
\begin{equation*}
E_{i}^{k*}=E_{i}^{k}+\sum_{j>i} p_{ij}E_{j}^{k*}
\end{equation*}
where $p_{ij}=1$ if $i<j\leq k$ and $C_{j}$ is proximate to $C_{i}$ and $p_{ij}=0$ in any other case. \\
We can now define two free $\mathbb{Z}-$modules with basis $\left\{E_{i}^{k*}\right\}_{i=1}^{k}$ and $\left\{E_{i}^{k}\right\}_{i=1}^{k}$ respectively, and construct the change of basis matrix $B_{k}$
\begin{equation}
B_{k}=\begin{pmatrix}
1 & 0 & \cdots & 0 & 0 \\
-p_{12} & 1 & \ddots & \vdots & \vdots \\
\vdots & -p_{23} & \ddots & \vdots & \vdots \\
\vdots & \vdots & \ddots & 1 & \vdots \\
-p_{1k} & -p_{2k} & \cdots & -p_{k-1k} & 1
\end{pmatrix}
\end{equation}
Note that its inverse $B_{k}^{-1}$ has also integer entries.
\end{rem}

Consider just one of the blow-ups conforming the sequence:
\begin{equation*}
\xymatrix{
E_{\alpha+1}^{\alpha+1}\ar[r]^{j_{\alpha+1}}\ar[d]_{g_{\alpha+1}} & Z_{\alpha+1}\ar[d]_{\pi_{\alpha+1}}\\
C_{\alpha+1}\ar[r]^{i_{\alpha+1}} & Z_{\alpha}
}
\end{equation*}
for $\alpha\in\left\{0,1,...,s-1\right\}$.

The following result show us the multiplication rules of $A^{\bullet}(Z_{\alpha+1})$ for any blow-up at some smooth center $C_{\alpha+1}$.

\begin{pro}\cite[Proposition 13.12.]{EisenbudHarris16}\label{ProGCRBU}
The Chow ring $A^{\bullet}(Z_{\alpha+1})$ is generated by $\pi_{\alpha+1}^{*}A^{\bullet}(Z_{\alpha})$ and $j_{\alpha+1*}A^{\bullet}(E_{\alpha+1}^{\alpha+1})$ as an algebra, that is, by classes pulled back from $Z_{\alpha}$ and classes supported on $E_{\alpha+1}^{\alpha+1}$. The rules for multiplication are the following:

\begin{numcases}
\pi_{\alpha+1}^{*}x\cdot\pi_{\alpha+1}^{*}y=\pi_{\alpha+1}^{*}(x\cdot y) & for\enspace $x,y\in A^{\bullet}(Z_{\alpha})$ \label{intpb} \\
\pi_{\alpha+1}^{*}x\cdot j_{\alpha+1*}t=j_{\alpha+1*}(t\cdot g_{\alpha+1}^{*}i_{\alpha+1}^{*}x) & for\enspace $x\in A^{\bullet}(Z_{\alpha}), t\in A^{\bullet}(E_{\alpha+1}^{\alpha+1})$ \label{intpbed} \\
j_{\alpha+1*}t\cdot j_{\alpha+1*}u=-j_{\alpha+1*}(t\cdot u\cdot\varsigma) & for\enspace $t,u\in A^{\bullet}(E_{\alpha+1}^{\alpha+1})$ \label{inted}
\end{numcases}

where $\varsigma=c_{1}(\mathcal{O}_{E_{\alpha+1}^{\alpha+1}}(1))$. \\
We identify $A^{\bullet}(Z_{\alpha})$ as a subring of $A^{\bullet}(Z_{\alpha+1})$ by means of the ring homomorphism
\begin{equation*} 
\pi_{\alpha+1}^{*}: A^{\bullet}(Z_{\alpha})\rightarrow A^{\bullet}(Z_{\alpha+1}),
\end{equation*}
\end{pro}

\begin{rem}
We will denote by $e_{i}^{\alpha*}$ (resp. $e_{i}^{\alpha}$) to be the class of $\left[E_{i}^{\alpha*}\right]$ (resp. $\left[E_{i}^{\alpha}\right]$) in $A^{1}(Z_{\alpha})$ for $i=1,...,\alpha$.
\end{rem}

In the next section we will give a presentation of the Chow ring $A^{\bullet}(Z_{\alpha})$ as a $\mathbb{Z}-$algebra of finite type.

\section{Main results}

Note that Proposition \ref{ProGCRBU} does not give a presentation of $A^{\bullet}(Z_{\alpha+1})$ as a $A^{\bullet}(Z_{\alpha})-$algebra, but only states the rules of multiplication. \\
If we could find generators of $A^{\bullet}(E_{\alpha+1}^{\alpha+1})$ as a $\mathbb{Z}-$algebra, $\left\{\gamma_{1},...,\gamma_{r}\right\}\in A^{\bullet}(E_{\alpha+1}^{\alpha+1})$, then 
\begin{equation*}
A^{\bullet}(Z_{\alpha+1})\cong A^{\bullet}(Z_{\alpha})\left[j_{\alpha+1*}(\gamma_{1}),...,j_{\alpha+1*}(\gamma_{r})\right]
\end{equation*}
would be a $A^{\bullet}(Z_{\alpha})-$algebra of finite type. One would like to have a presentation 
\begin{equation*}
A^{\bullet}(Z_{\alpha+1})\cong A^{\bullet}(Z_{\alpha})\left[w_{1},...,w_{r}\right]/\mathcal{J}
\end{equation*}
by sending $w_{i}$ to $j_{\alpha+1*}(\gamma_{i})$, with an explicit description of the ideal $\mathcal{J}$. The ideal $\mathcal{J}$will be computed in Theorems \ref{ThmCRTTPBU} and \ref{ThmCRPBUST}.
We will restrict ourselves to the case of sequences of point blow-ups, that is $C_{\alpha}=P_{\alpha}$, with the ground variety $Z_{0}\cong\mathbb{P}^{n}$. By \cite[Theorem 2.1.]{EisenbudHarris16}, $A^{\bullet}(Z_{0})\cong\mathbb{Z}\left[u\right]/(u^{n+1})$, by sending $u$ to $h$, where $h\in A^{1}(Z_{0})$ is the rational equivalence class of any hyperplane $\left[H\right]$ in $\mathbb{P}^{n}$, and $\forall\alpha$ $A^{\bullet}(E_{\alpha}^{\alpha})\cong\mathbb{Z}\left[w\right]/(w^{n})$ by sending $w$ to $\varsigma_{\alpha}$, with $\varsigma_{\alpha}\in A^{1}(E_{\alpha}^{\alpha})$ is the rational class of any hyperplane.

In the case of sequences of point blow-ups, we are able to give generators of the Chow ring $A^{\bullet}(Z_{s})$ as a $\mathbb{Z}-$algebra.

\begin{lem}\label{LemGCRPBU}
The Chow ring of the sky $A^{\bullet}(Z_{s})$ is generated by $\left\{h^{s*}, \left\{e_{i}^{s*}\right\}_{i=1}^{s}\right\}$ as a $\mathbb{Z}-$algebra.
\end{lem}

\begin{proof}
This follows by induction on $\alpha$. It is clear that $A^{\bullet}(Z_{0})$ is generated by $\left\{h\right\}$. Let us suppose that $A^{\bullet}(Z_{\alpha})$ is generated by $\left\{h^{\alpha*},\left\{e_{i}^{\alpha*}\right\}_{i=1}^{\alpha}\right\}$. Now by Proposition \ref{ProGCRBU} and due to the fact that $E_{\alpha+1}^{\alpha+1}\cong\mathbb{P}^{n-1}$, that is $A^{\bullet}(E_{\alpha+1}^{\alpha+1})\cong\mathbb{Z}\left[t\right]/(t^{n})$, by sending $t$ to $\varsigma_{\alpha+1}$, with $\varsigma_{\alpha+1}$ the rational equivalence class of any hyperplane in $\mathbb{P}^{n-1}$, and $e_{\alpha+1}^{\alpha+1*}\cdot e_{\alpha+1}^{\alpha+1*}=-j_{\alpha+1*}(\varsigma_{\alpha+1})$ by equation (4)  then $A^{\bullet}(Z_{\alpha+1})$ is generated by $\left\{h^{\alpha+1*},\left\{e_{i}^{\alpha+1*}\right\}_{i=1}^{\alpha+1}\right\}$ as a $\mathbb{Z}-$algebra.
\end{proof}

\begin{rem}
It makes sense then to define the augmented free $\mathbb{Z}-$modules with basis $\left\{H^{k*},\left\{E_{i}^{k*}\right\}_{i=1}^{k}\right\}$ and $\left\{\widetilde{H}^{k},\left\{E_{i}^{k}\right\}_{i=1}^{k}\right\}$ and the augmented change of basis matrix $B_{k}^{*}$
\begin{equation}
B_{k}^{*}=\begin{pmatrix}
1 &       0 & \cdots & \cdots & 0 & 0 \\
0 & 1 & 0 & \cdots & \vdots & \vdots \\
0 &-p_{12} & 1 & \ddots & \vdots & \vdots \\
\vdots & -p_{13} & -p_{23} & \ddots & \vdots & \vdots \\
\vdots & \vdots & \vdots & \ddots & 1 & \vdots \\
0 & -p_{1k} & -p_{2k} & \cdots & -p_{k-1k} & 1
\end{pmatrix}
\end{equation}
and its inverse $B_{k}^{*-1}$.
\end{rem}

\begin{thm}\label{ThmCRTTPBU}
The Chow ring of the sky $A^{\bullet}(Z_{s})$ is isomorphic to
\begin{equation}\label{EqCRTTPBU}
 A^{\bullet}(Z_{s})\cong\mathbb{Z}\left[x_{0},x_{1},...,x_{s}\right]/(\left\{x_{i}\cdot x_{j}\right\}_{\substack{i,j=0 \\ i\neq j}}^{s},\left\{(-1)^{n}(x_{i})^{n}+(x_{0})^{n}\right\}_{i=1}^{s}),
\end{equation}
by sending $x_{0}$ to the class $h^{s*}$ and $x_{i}$ to the class $e_{i}^{s*}$ for $i=1,...,s$.
\end{thm}

\begin{proof}
By lemma \ref{LemGCRPBU} there exist a exists a surjective morphism
\begin{equation*}
\phi: \mathbb{Z}\left[x_{0},x_{1},...,x_{s}\right]\rightarrow A^{\bullet}(Z_{s}),
\end{equation*}
 such that $\phi(x_{0})=h^{s*}$ and $\phi(x_{i})=e_{i}^{s*}$ for $i=1,...s$. Firstly we will prove that \\ $\mathcal{J}:=\left\langle \left\{x_{i}\cdot x_{j}\right\}_{\substack{i,j=0 \\ i\neq j}}^{s},\left\{(-1)^{n}(x_{i})^{n}+(x_{0})^{n}\right\}_{i=1}^{s}\right\rangle\subset Ker(\phi)$. To begin with, let us express the classes of the basis $\left\{E_{i}^{\alpha+1*}\right\}_{i=1}^{\alpha+1}$ in terms of the classes of the basis $\left\{E_{i}^{\alpha+1}\right\}_{i=1}^{\alpha+1}$, that is, since 
\begin{equation*}
e_{i}^{\alpha*}=e_{i}^{\alpha}+\sum_{j=i+1}^{\alpha}b_{j,i}e_{j}^{\alpha},
\end{equation*}
then
\begin{equation*}
e_{i}^{\alpha+1*}=e_{i}^{\alpha+1}+\sum_{j=i+1}^{\alpha}b_{j,i}e_{j}^{\alpha+1}+(\sum_{j=i}^{\alpha} p_{j\alpha+1}b_{j,i})e_{\alpha+1}^{\alpha+1}
\end{equation*}
where $b_{j,i}$ denotes the coefficients of the augmented change of basis matrix $B_{\alpha}^{*-1}$. \\
If we denote by $\varsigma_{\alpha+1}\in A^{1}(E_{\alpha+1}^{\alpha+1})$ the class of any hyperplane in $E_{\alpha+1}^{\alpha+1}$ then we have the following intersection products
\begin{numcases}
e_{\alpha+1}^{\alpha+1}\cdot e_{\alpha+1}^{\alpha+1}=-j_{\alpha+1*}(\varsigma_{\alpha+1}) & \label{selfint} \\
e_{j}^{\alpha+1}\cdot e_{\alpha+1}^{\alpha+1}=j_{\alpha+1*}(\varsigma_{\alpha+1}) & if\enspace $P_{\alpha+1}\rightarrow P_{j}$ \label{prox} \\
e_{j}^{\alpha+1}\cdot e_{\alpha+1}^{\alpha+1}=0 & otherwise \label{noprox}
\end{numcases}

where equation (\ref{selfint}) follows from equation (\ref{inted}) and equation (\ref{prox}) is a direct consequence of \cite[Corollary 6.7.1]{Fulton98}, that is $\pi_{\alpha+1}^{*}(e_{j}^{\alpha})=e_{j}^{\alpha+1}+e_{\alpha+1}^{\alpha+1}$, and equations (\ref{intpbed}) and (\ref{selfint}). So the following intersection product is $0$ 
\begin{equation}\label{EqRSTCRPBU}
(e_{j}^{\alpha+1}+p_{j\alpha+1}e_{\alpha+1}^{\alpha+1})\cdot e_{\alpha+1}^{\alpha+1}=0,
\end{equation}
and we can conclude that
\begin{equation*} 
e_{i}^{\alpha+1*}\cdot e_{\alpha+1}^{\alpha+1*}=(e_{i}^{\alpha+1}+\sum_{j=i+1}^{\alpha}b_{j,i}e_{j}^{\alpha+1}+(\sum_{j=i}^{\alpha} p_{j\alpha+1}b_{j,i})e_{\alpha+1}^{\alpha+1})\cdot e_{\alpha+1}^{\alpha+1}=0.
\end{equation*}
On the other hand $h^{\alpha+1*}\cdot e_{\alpha+1}^{\alpha+1}=0$ is a consequence of the moving lemma (see \cite[11.4 Moving lemma]{Fulton98}). If we make the pull back through $\pi_{s,\alpha+1}^{*}$ for all $\alpha$, then it follows that $\left\langle \left\{x_{i}\cdot x_{j}\right\}_{\substack{i,j=0 \\ i\neq j}}^{s}\right\rangle\subset ker(\phi)$. By \cite[Example 16.1.11]{Fulton98}, $A_{0}(Z_{0})$ is a birational invariant, that is $A_{0}(Z_{i})\cong\mathbb{Z}(h^{i*})^{n}$ for $i=1,...,s$, so since $(e_{\alpha+1}^{\alpha+1})^{n}=(-1)^{n-1}j_{\alpha+1*}(\varsigma_{\alpha+1}^{n})$ then $(e_{\alpha+1}^{\alpha+1})^{n}=(-1)^{n-1}(h^{\alpha+1*})^{n}$, and by making the pull back through $\pi_{s,\alpha+1}^{*}$ we conclude that $\left\langle \left\{(-1)^{n}(x_{i})^{n}+(x_{0})^{n}\right\}_{i=1}^{s}\right\rangle\subset Ker(\phi)$.\\
Now we will prove that $Ker(\phi)\subset\mathcal{J}$. Note that $\phi: \mathbb{Z}\left[x_{0},x_{1},...,x_{s}\right]\rightarrow A^{\bullet}(Z_{s})$ is homogenous, so $ker(\phi)$ is an homogenous ideal, and $\mathcal{J}$ is an homogenous ideal too by construction. Let us suppose that $P\left[x\right]\in Ker(\phi)/\mathcal{J}$ with $deg(P)=\eta$. Then $2\leq\eta\leq n$, since $\left\{x_{i}^{n+1}\right\}_{i=0}^{s}\in\mathcal{J}$, and $P\left[x\right]$ must be of the form $P\left[x\right]=\sum_{i=0}^{s}a_{i}x_{i}^{\eta}mod(\mathcal{J})$, since $\left\{x_{i}\cdot x_{j}\right\}_{\substack{i,j=0 \\ i\neq j}}^{s}\in\mathcal{J}$. Now if $\eta<n$, then $x_{i}^{n-\eta}P\left[x\right]$ will be also in $Ker(\phi)$, then $\phi(x_{i}^{n-\eta}P\left[x\right])=a_{i}(e_{i}^{s*})^{n}=0$, since $(e_{i}^{s*})^{n}\neq 0$ then $a_{i}=0$ for $i=0,1,...,s$. If $\eta=n$, since $\left\{(-1)^{n}(x_{i})^{n}+(x_{0})^{n}\right\}_{i=1}^{s}\in Ker(\phi)$ then it follows that $a_{0}+(-1)^{n+1}\sum_{i=1}^{s}a_{i}=0$, so $P\left[x\right]=0mod(\mathcal{J})$.
\end{proof}

\begin{rem}\label{RemMinGKer}
Note that $\left\langle x_{0},x_{1},...,x_{s}\right\rangle Ker(\phi)=\left\langle \left\{x_{i}x_{j}x_{k}\right\}_{\substack{i,j,k=0 \\ i\neq j \\ j\neq k}}^{s}, \left\{x_{i}^{n+1}\right\}_{i=0}^{s}\right\rangle$, so we have that \\
$Ker(\phi)/\left\langle x_{0},x_{1},...,x_{s}\right\rangle Ker(\phi)$ is a free $\mathbb{Z}-$module of finite rank $\binom{n+1}{2}+n$. Any set of generators of the ideal $Ker(\phi)$ is a set of generators of $Ker(\phi)/\left\langle x_{0},x_{1},...,x_{s}\right\rangle Ker(\phi)$ as $\mathbb{Z}-$module, so \\
$\left\{\left\{x_{i}\cdot x_{j}\right\}_{\substack{i,j=0 \\ i\neq j}}^{s},\left\{(-1)^{n}(x_{i})^{n}+(x_{0})^{n}\right\}_{i=1}^{s}\right\}$ is a minimal set of generators for $Ker(\phi)$.
\end{rem}

\begin{cor}
Given two sequences of point blow-ups $(Z_{0},...,Z_{s},\pi)$ and $(Z_{0}^{'},...,Z_{s^{'}}^{'},\pi^{'})$, if $s=s^{'}$ then $A^{\bullet}(Z_{s})\cong A^{\bullet}(Z_{s^{'}}^{'})$.
\end{cor}

\begin{proof}
It follows directly from equation $(\ref{EqCRTTPBU})$ in Theorem \ref{ThmCRTTPBU}.
\end{proof}

We can use $\left\{\widetilde{h}^{s},\left\{e_{i}^{s}\right\}_{i=1}^{s}\right\}$ as generators of the Chow ring $A^{\bullet}(Z_{s})$ as $\mathbb{Z}-$algebra instead.

\begin{thm}\label{ThmCRPBUST}
A presentation of $A^{\bullet}(Z_{s})$ using $\left\{\widetilde{h}^{s},\left\{e_{i}^{s}\right\}_{i=1}^{s}\right\}$ as generators is the following one:
\begin{equation}\label{EqSTCRPBU}
A^{\bullet}(Z_{s})\cong\frac{\mathbb{Z}\left[y_{0},y_{1},...,y_{s}\right]}{\mathcal{A}},
\end{equation}
where 
\begin{equation*}
\mathcal{A}=((\left\{y_{0}\cdot y_{i}\right\}_{i=1}^{s}, \left\{(y_{i}+\sum_{k=i+1}^{s}b_{k,i}y_{k})\cdot(y_{j}+\sum_{l=j+1}^{s}b_{l,j}y_{l})\right\}_{\substack{i,j=1 \\ i\neq j}}^{s},\left\{(y_{i})^{n}+((-1)^{n}+\#\left\{j\right\}_{j\rightarrow i})(y_{0})^{n}\right\}_{i=1}^{s}))
\end{equation*}
by sending $y_{0}$ to $\widetilde{h}^{s}$ and $y_{i}$ to $e_{i}^{s}$ for $i=1,...,s$.
\end{thm}

\begin{proof}
In this case there exists a surjective morphism
\begin{equation*}
\phi^{'}: \mathbb{Z}\left[y_{0},y_{1},...,y_{s}\right]\rightarrow A^{\bullet}(Z_{s})
\end{equation*}
with $\phi^{'}(y_{0})=h^{s*}$ and $\phi^{'}(y_{i})=e_{i}^{s}$ for $i=1,...,s$. Moreover we have the following commutative diagram
\begin{equation*}
\xymatrix{ \mathbb{Z}\left[x_{0},...,x_{s}\right]\ar[rd]^{\phi} & \\
\mathbb{Z}\left[y_{0},...,y_{s}\right]\ar[u]^{\rho}\ar[r]^{\phi^{'}} & A^{\bullet}(Z_{s})}
\end{equation*}
where $\rho: \mathbb{Z}\left[y_{0},...,y_{s}\right]\rightarrow\mathbb{Z}\left[x_{0},...,x_{s}\right]$ is the isomorphism induced by the augmentated change of basis matrix $B_{s}^{*}$, that is $\rho(y_{0})=x_{0}$ and $\rho(y_{i})=x_{i}-\sum_{j=i+1}^{s}p_{ij}x_{j}$. Now, by considering the following images through $\rho$:

\begin{equation*}
\begin{cases}
\rho((y_{i})^{n}+((-1)^{n}+\#\left\{j\right\}_{j\rightarrow i})(y_{0})^{n}) & =(x_{i}-\sum_{k=i+1}^{s}p_{ik}x_{k})^{n}+((-1)^{n}+\#\left\{j\right\}_{j\rightarrow i})(x_{0})^{n} \\
 & =(x_{i})^{n}+(-1)^{n}\sum_{k=i+1}^{s}p_{ik}(x_{k})^{n}+((-1)^{n}+\#\left\{j\right\}_{j\rightarrow i})(x_{0})^{n}+ \\
 & \sum\limits_{\substack{n_{i}+n_{i+1}+...+n_{s}=n \\ n_{i},...,n_{s}\neq n}}(-1)^{n-n_{i}}\binom{n}{n_{i},n_{i+1}...,n_{s}}\prod_{\beta=i}^{s} (p_{i\beta}x_{\beta})^{n_{\beta}} \\
 & =(-1)^{n}((-1)^{n}(x_{i})^{n}+(x_{0})^{n})+\sum_{k=i+1}^{s}p_{ik}((-1)^{n}(x_{k})^{n}+(x_{0})^{n})+ \\
 & \sum\limits_{\substack{n_{i}+n_{i+1}+...+n_{s}=n \\ n_{i},...,n_{s}\neq n}}(-1)^{n-n_{i}}\binom{n}{n_{i},n_{i+1}...,n_{s}}\prod_{\beta=i}^{s} (p_{i\beta}x_{\beta})^{n_{\beta}} \\
\rho(y_{0}\cdot y_{i}) & =x_{0}\cdot(x_{i}-\sum_{k=i+1}^{s}p_{ik}x_{k}) \\
 & =x_{0}\cdot x_{i}-\sum_{k=i+1}^{s}p_{i,k}x_{0}\cdot x_{k},\\
\rho((y_{i}+\sum_{k=i+1}^{s}b_{k,i}y_{k})\cdot(y_{j}+\sum_{l=j+1}^{s}b_{l,j}y_{l})) & =x_{i}\cdot x_{j}
\end{cases}
\end{equation*}
we can conclude that $\mathcal{A}\subset Ker(\phi^{'})$. \
The inclusion $Ker(\phi^{'})\subset\mathcal{A}$ is straightforward by Remark \ref{RemMinGKer}.
\end{proof}

\begin{defn}\label{DefFD}
We define for a given a sequence of blow-ups $(Z_{0},...,Z_{s},\pi)$ that an irreducible component $E_{i}$ is final if and only if $\nexists j$ such that $P_{j}$ is proximate to $P_{i}$. 
\end{defn}

\begin{rem}
In \cite{CamazonEncinas22} the definition of final divisor is different, but in the case of sequences of point blow-ups both definitions are equivalent.
\end{rem}

Also, in \cite{CamazonEncinas22}, a characterization of final was expressed in terms of the intersection of the irreducible components of the exceptional divisor. We can now use Theorem \ref{ThmCRPBUST} to refine the characterization.

\begin{cor}
$E_{i}$ is final if and only if its class in $A^{1}(Z_{s})$, that is $e_{i}^{s}$, satisfies the following two conditions
\begin{numcases} 
(e_{i}^{s})^{n}=(-1)^{r}(e_{i}^{s})^{n-r}(e_{j}^{s})^{r} \label{cond1} \\
(e_{j}^{s})^{n-1}e_{i}^{s}=(h^{s*})^{n} \label{cond2}
\end{numcases}

for every $j$ such that $e_{i}^{s}\cdot e_{j}^{s}\neq 0$. 
\end{cor}

\begin{proof}
If $E_{i}$ is final then $\nexists k$ such that $P_{k}$ is proximate to $P_{i}$. By equation (\ref{EqRSTCRPBU}) $(e_{j}^{i}+e_{i}^{i})\cdot e_{i}^{i}=0$ if $P_{i}$ is proximate to $P_{j}$ and $e_{i}^{i}\cdot e_{j}^{i}=0$ otherwise. Since $E_{i}$ is final then it follows that

\begin{numcases} 
(e_{j}^{s}+e_{i}^{s})\cdot e_{i}^{s}=0 & if $P_{i}\rightarrow P_{j}$ \label{numfinprox} \\
e_{i}^{s}\cdot e_{j}^{s}=0 & otherwise
\end{numcases}

From equation (\ref{numfinprox}) we can deduce that $(e_{i}^{s})^{n}=(-1)^{r}(e_{i}^{s})^{n-r}(e_{j}^{s})^{r}$. Moreover $(h^{s*})^{n}=(-1)^{n+1}(e_{i}^{s*})^{n}$, so $(h^{s*})^{n}=(-1)^{2n}e_{i}^{s}(e_{j}^{s})^{n-1}=e_{i}^{s}(e_{j}^{s})^{n-1}$. \\
Now we will prove that if $E_{i}$ is not final, then some of the above conditions fails. Among all the index $\left\{\beta\right\}$ satisfying $P_{\beta}\rightarrow P_{i}$ there must exist an index $j$ such that $P_{j}\rightarrow P_{i}$ but that there not exists $k$ with $P_{k}\rightarrow P_{i}$ and $P_{k}\rightarrow P_{j}$. Since $E_{j}^{j}$ is final for the sequence $(Z_{0},...,Z_{j},\pi_{j,0})$, then $(e_{j}^{j})\cdot(e_{i}^{j})^{n-1}=(h^{j*})^{n}$ and $(e_{i}^{j})^{n-1-\beta}(e_{j}^{j})^{1+\beta}=(-1)^{\beta}(e_{i}^{j})^{n-1}e_{j}^{j}$. Moreover, since $\nexists P_{k}$ with $P_{k}$ proximate to both $P_{i}$ and $P_{j}$,  then we can conclude that $(e_{j}^{s})\cdot(e_{i}^{s})^{n-1}=(h^{s*})^{n}$ and $(e_{i}^{s})^{n-1-\beta}(e_{j}^{s})^{1+\beta}=(-1)^{\beta}(e_{i}^{s})^{n-1}e_{j}^{s}$.
If $n$ is even, although $(e_{j}^{s})^{n-1}e_{i}^{s}=(h^{s*})^{n}$ since $n-2$ is even too, $(e_{i}^{s})^{n}\neq(-1)^{n-1}(e_{i}^{s})(e_{j}^{s})^{n-1}$ since by Theorem \ref{ThmCRPBUST} $(e_{i}^{s})^{n}=-(1+\#\left\{\beta\right\})(h^{s*})^{n}$ with $\#\left\{\beta\right\}\geq 1$ so condition (\ref{cond1}) fails. \\
If $n$ is odd, $(e_{j}^{s})^{n-1}e_{i}^{s}=-(h^{j*})^{n}$, since $n-2$ is odd too, so condition \ref{cond2} fails.
\end{proof}

\textbf{Some comments about the Chow ring of blow-ups at more general centers}

The main result in the literature about the structure of the Chow ring of a blow-up at a center of arbitrary dimension is the following one
\begin{thm}\cite[Appendix Theorem 1.]{Keel92}\label{ThmCRBUKeel}
Suppose the map of bivariant rings
\begin{equation*}
i_{\alpha+1}^{*}: A^{\bullet}(Z_{\alpha})\rightarrow A^{\bullet}(C_{\alpha+1})
\end{equation*}
is surjective, then $A^{\bullet}(Z_{\alpha+1})$ is isomorphic to
\begin{equation*}
A^{\bullet}(Z_{\alpha})\left[T\right]/(P(T),(T\cdot Ker(i_{\alpha+1}^{*}))),
\end{equation*}
where $P(T)\in A^{\bullet}(Z_{\alpha})\left[T\right]$ is any polynomial whose constant term is $\left[C_{\alpha+1}\right]$ and whose restriction to $A^{\bullet}(C_{\alpha+1})$is the Chern polynomial of the normal bundle $\mathcal{N}_{C_{\alpha+1}/Z_{\alpha}}$ i.e.
\begin{equation*}
i_{\alpha+1}^{*}(P(T))=t^{d}+c_{1}(\mathcal{N}_{C_{\alpha+1}/Z_{\alpha}})T^{d-1}+\cdots+c_{d-1}(\mathcal{N}_{C_{\alpha+1}/Z_{\alpha}})T+c_{d}(\mathcal{N}_{C_{\alpha+1}/Z_{\alpha}}),
\end{equation*}
(where $d=codim(C_{\alpha+1},Z_{\alpha})$). This isomorphism is induced by
\begin{equation*}
\pi_{\alpha+1}^{*}: A^{\bullet}(Z_{\alpha})\rightarrow A^{\bullet}(Z_{\alpha+1})
\end{equation*}
and by sending $-T$ to the class of the exceptional divisor.
\end{thm}
However, the surjectivity hypothesis of the theorem is quite restrictive. For example, let us consider the blow-up of a rational curve $C_{1}\subset Z_{0}\cong\mathbb{P}^{3}$ of degree $\gamma>1$.
 
\begin{equation*}
\xymatrix{
E_{1}^{1}\ar[r]^{j_{1}}\ar[d]_{g_{1}} & Z_{1}\ar[d]_{\pi_{1}}\\
C_{1}\ar[r]^{i_{1}} & Z_{0}
}
\end{equation*}

Note that in this case the restriction map $i_{1}^{*}: A^{\bullet}(Z_{0})\rightarrow A^{\bullet}(C_{1})$ is not surjective since $i_{1}^{*}(h)=\gamma\left[P\right]$, where $\left[P\right]$ denotes the class of a point $P\in C_{1}$, so we can not apply Keel formula of  theorem \ref{ThmCRBUKeel}. $A^{\bullet}(Z_{1})$ is no longer generated by $\left\{h^{1*}, e_{1}^{1*}\right\}$ and we need to add an extra generator $r_{1}=j_{1*}\left[g_{1}^{-1}(P)\right]$, which geometrically is the fiber of a point in the blow-up. \\
One can prove that 
\begin{equation*}
A^{\bullet}(Z_{1})\cong\mathbb{Z}\left[x_{0},x_{1},w_{1}\right]/\mathcal{I}
\end{equation*}
where 
\begin{equation*}
\mathcal{I}=((x_{0})^{2}\cdot x_{1}, x_{0}\cdot x_{1}-\gamma w_{1}, (x_{1})^{2}-c_{1}(N)w_{1}+\gamma(x_{0})^{2}, x_{0}\cdot w_{1}, (x_{0})^{3}+x_{1}\cdot w_{1})
\end{equation*}
by sending $x_{0}, x_{1}$ and $w_{1}$ to $h^{1*}, e_{1}^{1*}$ and $r_{1}$ respectively.

\bibliographystyle{plain}
\bibliography{biblioUVA}

\end{document}